\documentclass[12pt]{article}
\usepackage[utf8]{inputenc}
\usepackage{amssymb}
\usepackage{graphics}
\usepackage{amscd}
\usepackage{epsfig}
\usepackage{psfrag}
\usepackage{amsmath,amsthm}
\usepackage{amsfonts}
\newtheorem{theorem}{Theorem}
\newtheorem{lemma}[theorem]{Lemma}
\newtheorem{prop}[theorem]{Proposition}
\newtheorem{corr}[theorem]{Corollary}
\newtheorem{remark}[theorem]{Remark}
\newtheorem{deff}[theorem]{Definition}
\newtheorem{example}[theorem]{Example}

\title{Graph homology and graph configuration spaces}
\author{Vladimir Baranovsky, Radmila Sazdanovi\'c}
\date{}

\begin{document}
\maketitle

\begin{abstract}
If $R$ is a commutative ring, $M$ a compact $R$-oriented manifold
and $G$ a finite graph without loops or multiple edges,
we consider the graph configuration space $M^G$ and a
Bendersky-Gitler type spectral sequence converging to the
homology $H_*(M^G, R)$. We show that its
$E_1$ term is given
by the graph cohomology complex
$C_A(G)$
of the graded commutative
algebra $A = H^*(M, R)$ and its higher differentials are obtained
from the Massey products of $A$, as conjectured by Bendersky and
Gitler for the case of a complete graph $G$. Similar results apply
to the spectral sequence constructed from an arbitrary finite graph $G$
and a graded commutative DG algebra $\mathcal{A}$.
\end{abstract}

\section{Introduction}

Let $A$ be a graded algebra over a commutative Noetherian ring $R$ of
finite $Ext$ dimension. We assume that $A$ is a projective $R$-module.

For any finite graph $G$,
we will define the graph cohomology complex $C_A(G)$,
inspired by the construction of L.~Helme-Guizon and Y.~Rong, see for instance \cite{HR}
or Section $2$ of \cite{HPR}.
 To that end, let $E(G)$ and $V(G)$
be the sets of edges and vertices, respectively, and choose a bijection  of
$V(G)$ with $\{1, \ldots, n\}$, i.e. an enumeration of the vertices. This gives
an orientation for any edge $\alpha \in E(G)$: if $\alpha$  connects vertices
$i$ and $j$ with $i \leq j$ we write $\alpha: i \to j$.
For any subset $s \subset E(G)$ let $l(s)$ be the
number of connected components in the subgraph $[G:s]$ which has the same
set of vertices as $G$ but the edges in $s$ only.

Denote by $\Lambda = \Lambda(e_\alpha)$ the exterior algebra over $R$
on the generators $e_\alpha$, $\alpha \in E(G)$.
For $s \subset E(G)$ set
$e_s$ to be the exterior product of all $e_\alpha$, $\alpha \in s$, ordered with respect to
the lexicographic ordering on the pairs $(i,j)$ coming from the
edges $\alpha: i \to j$. Similarly, the connected components of
$[G:s]$ are naturally ordered by the smallest vertex contained in
a component.

Now define the bigraded complex $C_A(G)$ to be the quotient algebra of
$\Lambda \otimes_R A^{\otimes n}$ by the relations
$$
  e_{\alpha}  \otimes (a[i] - a[j]), \qquad a \in A, E(G) \ni \alpha: i \to j;
$$
where we denote by $a[i]$ the element $1^{\otimes (i-1)} \otimes a \otimes 1^{\otimes (n-i)} \in A^{\otimes n}$
for $i \in \{1, \ldots, n\}$ .

The complex $C_A(G)$ has a natural bigrading in which each $e_\alpha$ has bidegree
$(0, 1)$, and $a_1 \otimes \ldots \otimes a_n \otimes 1$
has bidegree $(\displaystyle{\sum^n_{i=1} \deg_A(a_i)}, 0)$.
The differential $\delta$ on $C_A(G)$ of bidegree $(0, 1)$ is given by the
wedge product with $\displaystyle{\sum_{\alpha \in E(G)} e_{\alpha}}$, see
page 428 of \cite{BG}.

\bigskip
\noindent
Alternatively, we can define $C_A(G)$ in terms of subgraphs in $E(G)$ as
the complex of projective $R$-modules
$$
C_A(G) := \bigoplus_{s \subset E(G)} e_s \cdot A^{\otimes l(s)} ,
$$
where $e_s \cdot a_1 \otimes \ldots \otimes a_{l(s)} $ has
bidegree $(\displaystyle{\sum^{l(s)}_{i=1} \deg_A a_i}, |s|)$, with
 $\delta$  acting as follows
\begin{eqnarray*}
\delta(e_s \cdot a_1 \otimes \ldots \otimes a_{l(s)})=
 \sum_{\begin{smallmatrix}
\alpha \in E(G)\\
l(s\cup \alpha) = l(s)
\end{smallmatrix}}  e_\alpha e_s \cdot a_1 \otimes \ldots \otimes a_{l(s)}  +  \\
 \sum_{\begin{smallmatrix}
\alpha \in E(G)\\
l(s\cup \alpha) = l(s)-1
\end{smallmatrix}}
 (-1)^{\tau} e_\alpha e_s \cdot a_1 \otimes \ldots  \otimes a_{t(\alpha)} a_{h(\alpha)}
\otimes \ldots \otimes a_{l(s)}
\end{eqnarray*}
where $s \cup \alpha$ is the subset obtained by adding $\alpha: i \to j$ to $s$
(we can assume that $\alpha \notin s$ as otherwise $e_s e_\alpha = 0$),
and $t(\alpha)$ and $h(\alpha)$ are the numbers of the connected components
in the subgraph $[G:s]$ containing $i$ and $j$, respectively.
The first sum corresponds to the case $h(\alpha) = t(\alpha)$ and the second to
$h(\alpha) \neq t(\alpha)$.
Note that $i$ and $j$ may not be the smallest vertices in their connected components
and one can have $h(\alpha) < t(\alpha)$ or
$t(\alpha) < h(\alpha)$. Depending on that, the product
 $a_{t(\alpha)} a_{h(\alpha)}$ has
either $(t(\alpha) - 2)$ or $(t(\alpha) - 1)$ terms to the left of it.
The sign $(-1)^\tau$ in the second group of terms is
the Koszul sign of the permutation of $a_1, \ldots, a_{l(s)}$ which moves
$a_{h(\alpha)}$
to the immediate right of $a_{t(\alpha)}$, and preserves the order of other
elements.

\bigskip
\noindent
Now let $M$ be a simplicial complex.
For any $\alpha: i \to j$
let $Z_\alpha$ be the diagonal in the cartesian product $M^{n}$ defined by
$m_i = m_j$,
and set
$$
Z_G = \bigcup_{\alpha \in E(G)} Z_\alpha, \qquad M^G = M^{n}
\setminus Z_G.
$$
We will call $M^G$ the \textit{graph configuration space} of $M$. In the case
when $G$  is the complete graph on $n$ vertices we get the classical
configuration spaces of ordered $n$-tuples of pairwise distinct points in $M$.

\begin{theorem} Assume that the cohomology algebra
$A = H^*(M, R)$ is
a projective $R$-module and that $G$ has no loops or multiple edges.
There exists a spectral sequence with $E_1$ term isomorphic
to $C_G(A)$ which converges to the relative cohomology
$H^*(M^{\times n}, Z_G; R)$.
\end{theorem}

\begin{remark}
This theorem resolves the conjecture of M. Khovanov, that there exists a spectral sequence from chromatic graph cohomology defined by L. Helme-Guizon and Y. Rong \cite{HR} to  Eastwood-Hugget graph homology \cite{EH}.
A standard consequence of the theorem is equality of the Poincare polynomials
(with respect to the total grading)
for $C_A(G)$ and $H^*(M^{\times n}, Z_G; R)$.
\end{remark}
\begin{remark}
For a general graph $G$ and an edge $\alpha \in E(G)$
one can use the deletion-contraction sequence
$$
0 \to C_A(G/\alpha) \to C_A(G) \to C_A(G \setminus \alpha) \to 0
$$
to compute the graph cohomology. There $G/\alpha$ is the
graph obtained by contracting $\alpha$ to a single vertex,
and $G \setminus \alpha$ is obtained by removing $\alpha$.
Then it is easy to see that the graph cohomology is zero when $G$ has
a loop, and it does not change if multiple edges $i \to j$ get
replaced by a single edge.
\end{remark}
\begin{remark}
   When $M$ is
a compact $R$-orientable manifold of dimension $m$, the
relative cohomology groups $H^*(M^{\times n}, Z_G; R)$
are isomorphic to the homology groups $H_{mn-*} (M^G; R)$ by
Lefschetz duality. Observe, however, that for existence of the spectral
sequence we still have to assume that $H^*(M, R)$ is projective over
$R$ (one of the reasons is that we use the Kunneth formula for cohomology).
In general the cohomology algebra $A$ needs to be replaced
by an appropriate projective DG-algebra $\mathcal{A}$ resolving it,
and $C_A(G)$ by $C_{\mathcal{A}}(G)$, as in Section 2.2 below.
\end{remark}

\bigskip
\noindent
In Section 4 we study the higher differentials of this spectral sequence
and show that they are determined by the matrix
 Massey products of $A$,
as conjectured by Bendersky and Gitler in \cite{BG}. Our main results
here are Proposition \ref{massey}, explaining how the $A_\infty$-algebra
structure on $A$ and an application of perturbation theory to
spectral sequences (as recalled in Proposition \ref{standard}) allow us
 to compute the spectral sequence differentials;
and Proposition \ref{degeneration} which says that the spectral sequence
degenerates starting with the page $E_m$,
 where $m$ is the number of vertices in the
largest subtree of $G$. Also, a standard argument shows that in some cases,
e.g. when $M$ is a compact K\"ahler manifold, the spectral sequence
degenerates in the $E_2$ term.

\bigskip
\noindent
\textbf{Acknowledgements.} The first author wants to thank HSE in Moscow,
Russia and USTC in Hefei, China where he stayed while working on this paper,
for their excellent research conditions. The second author wants to thank
the organizers of the Special Program ``Homology theories for knots and links"
for the opportunity to do research at MSRI where this collaboration started, and
 NSF 0935165 and AFOSR FA9550-09-1-0643 grants for support towards the end of the project.

\section{Spectral sequences}

\subsection{Proof of Theorem 1}

For any simplicial topological space $X$,
we denote by $C^*(X; R)$ its
cochain complex.  Suppose that $Z \subset X$ is a subspace which
is a union of closed subspaces $Z_{\alpha}$, $\alpha \in E$, where
$E$ is a finite ordered set.
For a finite subset $s \subset E$ let
$$ Z_s = \displaystyle{\bigcup_{\alpha \in s} Z_{\alpha}}$$
and denote $Z_\emptyset = X$ for notational convenience.
By pages $425$ and $427$ of \cite{BG}, the relative cohomology $H^*(X, Z; R)$
can be computed as the total cohomology of a bicomplex
\begin{equation}
\label{cochains}
C^*(Z_\emptyset, R) \to \bigoplus_{\alpha \in E} C^*(Z_\alpha; R)
\to \bigoplus_{s \subset E; |s|=2} C^*(Z_s; R) \to \ldots
\end{equation}
where the differential comes from the obvious simplicial structure
on the collection of subsets in $E$.
Applying one of the two standard spectral sequences of a bicomplex
we obtain a spectral sequence converging  to $H^*(X, Z; R)$ with
$$ E_1^{pq} = \bigoplus_{s \subset E; |s| = p}  H^q(Z_s; R),$$
and the differential $\partial_1: E_1^{p, q} \to E_1^{p+1, q}$
is the usual simplicial differential constructed from the
pullbacks with respect to the closed embeddings
$$
Z_t = Z_s \cap Z_\alpha \subset Z_s; \qquad t = s \cup \{\alpha\},
\alpha \notin s.
$$
Next, we specialize to the case when $X = M^n$, $G$ has no
loops or multiple edges and
$Z = \displaystyle{\bigcup_{\alpha \in E(G)} Z_e}$ comes from the set $E = E(G)$
of edges in $G$. For a general subset $s \subset E$ the space
$Z_s$ can be identified with $M^{l(s)}$. By projectivity
 (and hence flatness) of $A$, the Kunneth formula applies to give
$$
H^*(Z_s, R) = A^{\otimes l(s)}
$$
where the tensor product is taken over $R$.

To compute the differentials explicitly, consider two cases.
Firstly, the unique element $\alpha \in t \setminus s$ may connect two
vertices within the same connected component of $[G:s]$. In this case
the embedding $Z_t \subset Z_s$ is an isomorphism and hence
induces the identity map on cohomology.

Secondly, $\alpha$ may connect two of the $l = l(s)$ connected
components of the graph $[G:s]$. To simplify notation assume that
these are the components corresponding to the first two factors of
$A^{\otimes l}$. Then the embedding $Z_t \to Z_s$ is the product of
the diagonal map $M \to M \times M$ and the identity on
the other $M^{l-2}$ factors. But the pullback with respect to the
diagonal map induces on cohomology
precisely the cup product $A \otimes A \to A$. Hence
$Z_t \subset Z_s$in this case induces the map
$$
A^{\otimes l} \to A^{\otimes (l-1)}, \qquad
a_1 \otimes a_2 \otimes a_2 \otimes \ldots \otimes a_l
\mapsto a_1 a_2 \otimes a_3 \otimes \ldots \otimes a_l.
$$
This finishes the proof of the theorem. $\square$

\bigskip
\noindent
\begin{remark}
   One can give a slight generalization when there
is a continuous map $f: N \to M$
which makes $B= H^*(N, R)$ into a module over $A=H^*(M, R)$.
In this case we assume that the set of vertices in the
graph $G$ is $I =\{0, \ldots, n\}$ and define the generalized
configuration space $M^{G, f}$ as the open
complement $(N \times M^{\times n}) \setminus Z_G$.
When the edge $\alpha$ connects two nonzero vertices the subset
$Z_\alpha$ is still defined by the condition $m_i = m_j$
while for an edge connecting $i=0$ and $j$ we use the
condition $f(n) = m_j$. As with $A$, we need to assume
that $B$ is projective over $R$.

Since the graph embedding $(Id_N, f): N \to N \times M$
induces the module action map $B \otimes_R A \to
B$ on cohomology, we get the spectral sequence
with $E_1$ given by the graph cohomology of the pair
$(A, B)$ (the complex is constructed similarly, but the $A$-module $B$
is placed at the zero vertex), converging to the
relative cohomology $H^*(N \times M^n, Z_G; R)$.
When $M, N$ are compact $R$-orientable manifolds and $f$
is a smooth map, this
is isomorphic to
the homology of the generalized configuration space
$M^{G, f}$.
\end{remark}

\subsection{Graph cohomology of DG algebras.}

Let $\mathcal{A}$ be a commutative  DG algebra. Then the complex $C_{\mathcal{A}}(G)$
has another differential $d$ of bidegree $(1, 0)$ induced by the
differential of $\mathcal{A}$, and it is easy to check that $d \delta + \delta d =0,
d^2=0$. Therefore we can consider the total differential $D= d + \delta$.

Now assume that $\mathcal{A}$ is a commutative DG algebra such that
the bicomplex $C_{\mathcal{A}}(G)$ may be connected with
the graph configuration space
bicomplex (\ref{cochains}) by a sequence of first quadrant
bicomplex quasi-isomorphisms, which restrict to quasi-isomorphisms along
the columns. Then the two bicomplexes have isomorphic spectral
sequences (starting with the $E_1$ term) associated with the vertical filtration,
as follows from the standard definitions, e.g. in \cite{GM}.

\begin{example}
   We can take $\mathcal{A} = H^*(M; R)$
with the zero differential if the space $M$ is $R$-formal, i.e.
if $H^*(M; R)$ and the cochain algebra $C^*(M; R)$ may be
connected with $H^*(M; R)$ by a chain of DG algebra
quasi-isomorphisms. When $R$ is the field $\mathbb{Q}$
of rational numbers, we can take the complex of Sullivan
cochains, and for the field $\mathbb{R}$ of real numbers
we can take the De Rham complex of differential forms.
Finally by a result of \cite{A}, for $R$ a field of
finite characteristic $p$,
 the DG algebra $\mathcal{A}$ exists if $M$ is $r$-connected
and $pr > \dim M.$
\end{example}

\begin{remark}
  When the algebra $A$ is not flat over $R$,   a better version
of its graph homology is obtained by taking a flat $DG$-resolution
$\mathcal{A} \to A$ and then computing the cohomology of $C_{\mathcal{A}}(G)$
with respect to the total differential.
\end{remark}

\begin{remark}
 Ideally, one should be able to work with the graph homology of
$C^*(M; R)$ itself. Although this DG algebra is definitely not graded
commutative, its deviation from commutativity can be measured by using the
$\cup_i$-operations of Steenrod, which were generalized by McClure and Smith
in \cite{MS} to certain \textit{``sequence operations"} on $C^*(M; R)$. We expect
that these operations (i.e. the $E_\infty$ algebra structure on
$C^*(M; R)$) can be used to adjust the total differential
$D = d + \delta$ by adding a term $d_2 + d_3 + \ldots$, with
$d_i$ of bidegree $(-i+1, i)$, so that
the total sum satisfies $D^2 = 0$. For instance, if $G$ has edges $(\alpha, \beta)$
with $t(\alpha) = t(\beta) = 1$, $h(\alpha) =2$, $h(\beta) = 3$ and $n=3$,
one can set $d_2(a_1 \otimes a_2 \otimes a_3) = a_1 (a_3 \cup_1 a_2) e_\alpha e_\beta$
to achieve $D^2=0$. For $n\geq 4$ one probably needs the sequence operations of
\cite{MS}. The resulting complex should compute the relative homology
$H^*(M^{\times n}, Z_G; R)$ for any pair $(M, R)$. We plan to return to this
topic in a future work.
\end{remark}

\section{Perturbation Lemma and applications.}

The material of this section is fairly standard, we collect it here for the reader's convenience
and also to fix the notation.

\subsection{Basic Perturbation Lemma}

\begin{deff}
  Let  $K, L$  be a pair of complexes with differentials
$d_K$, $d_L$ respectively.  Consider morphisms of complexes $f: K \to L$,
$g: L \to K$ such that $fg$ is equal to the identity $1_L$ on $L$ and
$1_K - gf = d_K h + h d_K$  for some homotopy $h$. The triple $(f, g, h)$ is
called a \textit{reduction} (or a \textit{strong deformation retract})
if in addition the following \textit{side conditions} are satisfied
\begin{equation}\label{side}
  hg = fh =  hh = 0.
\end{equation}
\end{deff}
\noindent
It is well known, cf. \cite{JL}, that  conditions (\ref{side})  can be ensured by
adjusting an arbitrary homotopy $h$: first replacing it by $h'= (dh + hd) h (dh + hd)$
which satisfies $h'g=fh'=0$, and then further setting $h''= h' d h'$ which
will imply all three side conditions.

\begin{remark}
   Suppose that $K$ is a complex of projective
$R$-modules, such that $L = H^*(K)$ is also projective, and that
the commutative ring $R$ has finite $Ext$ dimension. Denoting
by $B^n$, resp. $Z^n$ the coboundaries, resp. the cocycles, of $K$,
we have the standard exact sequences:
$$
0 \to B^n \to Z^n \to L^n \to 0; \qquad 0 \to Z^n \to K^n \to B^{n+1} \to 0.
$$
Since we assumed $L^n$ and $K^n$ to be projective, this gives
for any $R$-module $M$ the isomorphisms for $i \geq 1$ and all $n$:
$$
Ext^i_R(Z^n, M) \simeq Ext^i_R(B^n, M); \qquad
Ext^i_R(Z^n, M) = Ext^{i+1}_R(B^{n+1}, M).
$$
Since we assumed $R$ to be of finite $Ext$-dimension, by induction
on $k$ we get  that $Ext^i_R(B^n, M) = Ext^{i+k}_R(B^{n+k}, M)$ which must
be zero if $k$ is large enough. Therefore each $B^n$ is also projective and
we can choose splittings
$$
K^n \simeq B^{n+1} \oplus Z^n \simeq B^{n+1} \oplus B^n \oplus L^n
$$
such that the differential $K^n \to K^{n+1}$ is the composition of the
projection $K^n \to B^{n+1}$ and the embedding $B^{n+1} \to K^{n+1}$.
Then a reduction $(f, g, h)$ may be defined as follows: $f, g$ are the obvious
projection and embedding and $h$ is the ``inverse" composition
$K^{n+1} \to B^{n+1} \to K^n$.
\end{remark}

\noindent
Now suppose we have a perturbation $\widehat{d}_K =
d_K + \delta_K$
of the differential $d_K$
such that $\delta_K^2 = 0, d_K \delta_K + \delta_K d_K = 0$.
We assume in addition that
the composition $\delta_K H$ is \textit{locally nilpotent}, i.e. on any particular
$x \in K$ we have $(\delta_K h)^n x = 0$ where the positive integer $n$ may
depend on $x$.

The following result is known as the Basic Perturbation Lemma, see
\cite{JL} and references in that paper.

\begin{lemma}\label{BPL}
Under the above assumptions, there exist: a perturbation of the
differential $\widehat{d}_L
= d_L + \delta_L$ on $L$, morphisms of complexes $\widehat{f}: K \to L$,
$\widehat{g}: L \to K$ and a homotopy $\widehat{h}$ (with respect to
the perturbed differentials on $K, L$), given by the formulas
$$
\delta_L = fXg,\; \widehat{f} = f (1 - Xh), \; \widehat{g} = (1 - hX) g, \;
\widehat{h} = h - hXh;
$$
where
$$
X = \delta_K - \delta_K h \delta_K + (\delta_K h) ^2 \delta_K -
(\delta_K h)^3 \delta_K + \ldots\hspace{2cm} \square
$$

\end{lemma}

\subsection{Perturbations and Spectral Sequences}

Now we apply the previous result to give a very concrete realization of
the spectral sequence of a bicomplex of modules over
a ring, in the case when a splitting
homotopy is chosen for one of the differentials (say the vertical).
Although it is not easy to find an exposition of this approach in the
published literature (see \cite{RRS}, for instance) it is fairly
old and known to the experts in the field.

Consider a bicomplex with a
vertical $d: A^{p, q} \to A^{p, q+1}$ and a horizontal
differential $\delta: A^{p, q} \to A^{p+1, q}$.  We want to
identify the higher differentials of the standard spectral sequence
with $E_1^{p, q} = H_d (A^{p, q})$ converging to the
cohomology of the total complex $(K, d + \delta)$.

To that end,  let $L$ be the total complex of the $E_1$ term  and
assume there is a reduction of $A^{p, \cdot}$ to $ H_d(A^{p, \cdot})$
along each column of the original bicomplex.
This induces a reduction $(f, g, h)$ of the total complex $(K, d)$
onto $(L, 0)$.

Trying to compute the cohomology of $(K, d + \delta),$
we can treat $d + \delta$ as a perturbation of $d$ and apply
the Basic Perturbation Lemma \ref{BPL}. Observe that the local nilpotence
condition on $\delta h$ holds, for example, when $A^{p, q}$
is concentrated in the first quadrant (since $\delta$ moves an
element to the right, and $h$ moves it down).

By the Basic Perturbation Lemma, the complexes $(K, d + \delta)$
and $(L, f Xg)$ are homotopic; hence instead we can compute
the cohomology of $L$ with respect to
$$
\widehat{d}_L = d_1 + d_2 + d_3 + d_4 + \ldots
$$
where
\begin{equation}
\label{dee}
d_i = (-1)^{i-1} f (\delta h)^{i-1} \delta g.
\end{equation}
Each $d_i$ is an operator $E_1^{p, q} \to E_1^{p+i, q + 1 - i}$.

Since the homotopy $h$ preserves the filtration on $K$
the spectral sequences of  filtered complexes $K$ and $L$ agree,
see e.g. Theorem 15 in \cite{RRS}.
Writing out the standard definitions for the spectral sequence of
the filtered complex $(L, fXg)$ we get the following result.

\begin{prop}
\label{standard}
For every $i \geq 2$ an element
of $E^{p, q}_i$ is represented by $x \in L^{p, q}$ such that the
following system of equations on $x_2, \ldots, x_{i-1}$ admits a
solution
$$
d_1 (x) = 0; \qquad d_2(x) + d_1(x_2) = 0;
\qquad d_3(x) + d_2(x_2) + d_1(x_3) = 0; \ldots
$$
\begin{equation}
\label{system}
d_{i-1}(x) + d_{i-2}(x_2) + \ldots  + d_1(x_{i-1}) = 0,
\end{equation}
modulo the elements of the form $x = d_{i-1}(b_2) + \ldots
+ d_2 (b_{i-1})
+ d_{1}(b_i)$ where $b_{i}$ is arbitrary, and
$(b_2, \ldots, b_{i-2})$ satisfy a system of equations, obtained
from (\ref{system})
by setting $x=0$ and replacing $x_j$ by $b_j$.
The value of the differential $\partial_i: E_i^{p, q} \to E_i^{p+i, q + 1 - i}$
on such $x$ is represented  by the following element of $E_1^{p+i, q + 1 - i}$:
$$
d_i(x) + d_{i-1}(x_2) + \ldots + d_2(x_{i-1}).
$$
\end{prop}
\begin{corr}
Let $i \geq 2$ and suppose that $x \in E_1^{p, q}$ is such that
$d_1(x) = d_2(x) = \ldots = d_{i-1}(x) = 0$. Then such $x$ represents
a class in $E^{p, q}_i$ (as we may simply
take $x_i = 0$ for $i \geq 2$) and $\partial_i(x)$ is represented
by $d_i(x) \in E_1^{p+i, q-i+1}$.
\end{corr}

\subsection{A-infinity structures on cohomology and Massey products.}

Let $\mathcal{A}$ be a DG algebra, $A$ its cohomology algebra,
and assume there is a reduction
$(f, g, h)$ of $\mathcal{A}$  to $A$ as before.
Note that in general it may not be possible to choose either $f$ or $g$
multiplicative (e.g. if the derived categories of $\mathcal{A}$ and $A$
are not equivalent). In fact, $A$ admits a system of higher products
$m_i: A^{\otimes k} \to A$ which,  in a sense, measure how far the two
algebras are from being quasi-isomorphic as DG algebras, cf.
\cite{Ka}.

We recall this construction from the perturbation theory viewpoint.
First consider $\mathcal{A}$ and $A$ as non-unital algebras with the zero
product and consider their bar constructions $B(\mathcal{A})
= \displaystyle{\bigoplus_{k \geq 0} \mathcal{A}^{\otimes k}}$, and similarly
for $B(A)$. The differential on $\mathcal{A}$ extends by the Leibnitz rule
to $B(\mathcal{A})$ and the original contraction $(f, g, h)$ extends to
a contraction from $B(\mathcal{A})$ to $B(A)$. The extension for $f$ and
$g$  is obvious, and $h$ is given in $A^{\otimes n}$ by
$$ \sum_{i =1}^{n} (gf)^{\otimes (i-1)} \otimes h \otimes 1^{\otimes (n-i)}.$$
If we now recall the non-trivial product on $\mathcal{A}$, this will give
a perturbation $d + \delta$ of the initial differential $d$ on $K = B(\mathcal{A})$.
Hence by Perturbation Lemma \ref{BPL} we can write a new non-zero differential
on $L= B(A)$ such that the two bar constructions are still homotopic.
One can check that the new differential agrees with the
natural coproduct on $B(A)$ and it is therefore encoded by a
series of maps $m_k: A^{\otimes k} \to A$. One further checks that
$m_1=0$  and $m_2$ is the standard product $f(g(a) g(b))$.

Explicitly, one can define $m_n: A^{\otimes n} \to A$ for $n \geq 2$ using the
operations $\lambda_n: \mathcal{A}^{\otimes n} \to \mathcal{A}$
for $n \geq 2$ by setting $\lambda_2(a_1 \otimes a_2) = a_1 a_2$
and
\begin{equation}
\label{merkulov}
\lambda_n = \sum_{p=1}^{n-1} (-1)^{p+1} \lambda_2[h\lambda_p
\otimes h\lambda_{n-p}],
\end{equation}
where in the terms with $p = 1$ and  $n-1$ we formally set
$h \lambda_1 = - id_{\mathcal{A}}$. Then
\begin{equation}
\label{ainfty}
f \circ \lambda_n \circ g^{\otimes n}: A^{\otimes n} \to A
\end{equation}
for $n \geq 2$ gives an $A_\infty$-structure $\{m_n\}_{n \geq 2}$.
See \cite{JL}, \cite{LPWZ}, \cite{CG} and references therein.

The higher products $m_k$ for $k \geq 3$ in general depend on the choice of
$h$. However, for special choices of $a_1, \ldots, a_k$ the value
$m_k(a_1, \ldots, a_k)$ belongs to the coset of a Massey product,
see Theorem 3.1 in \cite{LPWZ}, see also Theorems 6.3 and 6.4 in
\cite{JL}, hence at least this coset is independent
of $h$ for this particular choice of the arguments.

\section{Higher differentials and Massey products.}

As in the previous subsection, any reduction $(f, g, h)$ of
$\mathcal{A}$ to $A$ induces a reduction of the graph cohomology
complex $C_G(\mathcal{A})$ onto $C_G(A)$. We will show that the
operators $d_i$ defined by (\ref{dee}) are completely determined
by the $A_\infty$-structure on $A$, induced by $(f, g, h)$. In view
of Proposition \ref{standard},
this gives information about the
higher differentials of the spectral sequence. In addition, specific
values of the $A_\infty$ operations are given by Massey products,
see below, confirming the conjecture by Bendersky and Gitler
(formulated originally for the complete graph $G$, i.e. the
usual configuration space).

For other situations in which differentials of a spectral sequence
are related to the (matric) Massey products see  Theorem 12.1 in \cite{S},
Corollary 4.6 in \cite{Ma} or \cite{La}.

\subsection{Computation of the $d_i$ operators.}

We want to compute the values of $d_i: E_1^{p, q} \to
E_1^{p+i, q + 1 - i}$. Write
$$
d_i(e_s \cdot a_1 \otimes \ldots \otimes a_{l(s)})
= \sum_{t| s \subset t} e_t \cdot C_i^{s \subset t}
(a_1 \otimes \ldots \otimes a_{l(s)})
$$
\begin{prop}
The value of $C_i^{s \subset t} (a_1 \otimes \ldots \otimes a_{l(s)})$ is zero
unless the set of edges $t \setminus s$ has $i$ elements and projects to
a tree with $i$ edges in the graph $G/s$ obtained by contracting
all edges in $s$. In the latter case, suppose that the edges of $t \setminus s$
connect the components $1 \leq \alpha_1 < \alpha_2 < \ldots < \alpha_{i+1} \leq l(s)$
in the graph $[G:s]$, then
\begin{eqnarray*}
\lefteqn{ C_i^{s \subset t} (a_1 \otimes \ldots \otimes a_{l(s)})=
(-1)^\varepsilon a_1 \otimes \ldots \otimes a_{\alpha_1 - 1} \otimes}\\
&&\hspace{-1cm}
\otimes C_i^{\emptyset \subset (t \setminus s)}
(a_{\alpha_1} \otimes \ldots \otimes a_{\alpha_{i+1}})
\otimes \ldots \widehat{a_{\alpha_2}} \ldots \widehat{a_{\alpha_{i+1}}}
\otimes \ldots \otimes a_{l(s)},\end{eqnarray*}
\noindent and $(-1)^\varepsilon$ is the sign of the permutatation
$$
(a_1, \ldots, a_{l(s)}) \mapsto
(a_1, \ldots, a_{\alpha_1 - 1}, a_{\alpha_1},
a_{\alpha_2}, \ldots, a_{\alpha_{i+1}}, \ldots
\widehat{a_{\alpha_2}} \ldots \widehat{a_{\alpha_{i+1}}}, \ldots,
a_{l(s)}).
$$
\end{prop}
\noindent
\begin{proof} The side conditions $hg=hh=0$ imply that when we evaluate $\delta
(h\delta)^{i-1}$ on a tensor monomial, all occurrences of $h$
should be applied only to the newly created products involved in the
definition of $\delta$. Also, the last occurrence of $h$ should be multiplied
by something else before we apply $f$ to it (since $fh = 0$).
Therefore, every tensor factor standing next to $e_t$ in
$f^{\otimes (l(s) - i)}
\delta (h\delta)^{i-1} g^{\otimes l(s)}(e_s \cdot a_1 \otimes \ldots a_{l(s)})$,
is either one of the original $a_i$,
or an expression involving $p$ multiplications and $\leq p-1$
uses of $h$. Since the total number of multiplications is $i$ and $h$
occurs $(i-1)$ times, exactly one of those factors can have the latter
form. This means that $i$ components of $[G:s]$ assemble into
a single component of $[G:t]$ and other components remain untouched,
as claimed. The rest follows from the contraction isomorphism of
complexes $e_s \cdot C_G(A) \simeq C_{G/s}(A)$.
\end{proof}

\medskip
\noindent
The previous proposition means that, in computing operators
$d_i$ (involved in the formulas for the spectral sequence
differentials $\partial_i$) we can reduce
to the case when $s = \emptyset$, $t = E(G)$,  and $G$ is a connected
tree with $i = n-1$ edges.
\begin{prop}
\label{massey}
Under the above assumptions
$$
d_{n-1} (a_1 \otimes \ldots \otimes a_n) = (-1)^{\frac{n(n-1)}{2}}
\sum_{\sigma \in \Sigma(G)} (-1)^\sigma e_t \cdot m_n \circ \sigma
$$
where $m_n: A^{\otimes n} \to A$ is the $n$-th product of the
$A_\infty$-structure on $A$ induced by the reduction of $\mathcal{A}$ onto
$A$ as in (\ref{ainfty}); and the sum runs over the set $\Sigma(G)$
of all permutations of $\{1, \ldots, n\}$ such that
the total order in which $\sigma(1) < \sigma(2) < \ldots < \sigma(n)$
refines the partial order generated by $i < j$ whenever
there is an edge $i \to j$ in the graph $G$.
\end{prop}
\noindent
Note that the action of  $\sigma: A^{\otimes n} \to A^{\otimes n}$
involves an appropriate Koszul sign.
\begin{figure}\begin{center}
  \scalebox{0.8}{\includegraphics{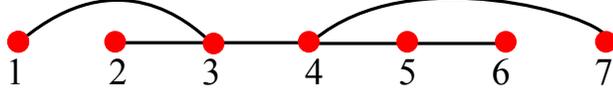}}
 \caption{An example for $n=7$.}
\end{center}
\label{Linear1223}
\end{figure}
\begin{remark} See Theorem 12.1 on page 60 of \cite{S} for a similar
(and perhaps related) situation when a spectral sequence differential
is related to an $A_\infty$ structure. See also \cite{FT} for a special
case (with $n=3$) of the above result.
\end{remark}
\begin{example}
In the example shown on Figure 1.  the possible
permutations $\sigma$ are given by
$$
(1234567), (2134567), (1234576), (2134576), (1234756), (2134756).
$$
\end{example}

\medskip
\noindent
\begin{proof} We use induction on $n$. For $n=2$ the graph $G$ consists of
a single arrow and the assertion easily follows from the definition of $d_1$.
For general $n$ let us consider the last edge $\alpha$ to disappear when
we apply the leftmost term in the expression $\delta (h\delta)^{n-2}$ to
the element $g^{\otimes n} (a_1 \otimes \ldots \otimes a_n)$. When
we remove $\alpha: i \to j$ the graph $G$ splits into a disjoint union of two trees.
We have three cases. First, one of this trees may consist of a single vertex
$i$. Denoting the other tree by $G_1$ we encode this case
by the diagram $i \to G_1$. Next, both trees may contain at least two vertices.
Denoting these trees by $G_2$ and $G_3$ we will use the notation
$G_2 \to G_3$. Finally, if one of the trees is the single vertex $j$
and the other is denoted by $G_4$ we encode this situation by $G_4 \to j$.

Denote by $\lambda_G$ the expression $\delta (h\delta)^{n-2}$.
We would like to show that $\lambda_G$ is equal to the alternating sum of
$\lambda_n \circ \sigma$, $\sigma \in \Sigma_n$, modulo the image of
$h$ (which does not affect the value of $d_i$ due to $fh =  0$).
It is clear
that looking at ``the last edge to be used with $\delta$" we get an inductive formula
\begin{equation}
\label{induct}
\lambda_G \equiv \sum_{i \to G_1} \pm \lambda_2[h\lambda_1 \otimes h \lambda_{G_1}]
+ \sum_{G_2 \to G_3} \pm \lambda_2[h\lambda_{G_2} \otimes h\lambda_{G_3}]
+ \sum_{G_4 \to j} \pm \lambda_2[h\lambda_{G_4} \otimes h \lambda_1]
\end{equation}
modulo terms in the image of $h$.
We would like to establish the inductive step by applying the
antisymmetrization in $\sigma$ to the formula (\ref{merkulov})
and comparing the result with the above recursive formula.
The first terms matches the $p=1$ term in the antisymmetrization of
(\ref{merkulov}) since the vertices $i$ which can occur in
$i \to G_1$ are exactly the vertices which can occur as $\sigma(1)$, and
for the second factor we can apply the inductive assumption to $G_1$.
Similarly the third term above matches the $s = p-1$ term in the
antisymmetrization of (\ref{merkulov}) since the possible values of $\sigma(n)$
are exactly the vertices $j$ which have a single edge coming into it.

Hence it remains to compare the second term above and the terms corresponding
to $2 \leq p \leq n-2$ in the formula (\ref{merkulov}).
For the latter terms, consider $\sigma \in \Sigma_G$, then we want to
understand $h\lambda_p(a_{\sigma(1)} \otimes \ldots a_{\sigma(p)})
\otimes h\lambda_{n-p}(a_{\sigma(p+1)} \otimes \ldots \otimes a_{\sigma(n)})$
This appears in the middle term of (\ref{induct}) precisely when
both subsets of vertices $\sigma(1), \ldots, \sigma(p)$
and $\sigma(p+1), \ldots, \sigma(n)$ give connected subgraphs
$G_2$ and $G_3$, respectively.

We would like to show that all other terms sum to zero. We group together
those terms that give fixed $G_2$ and $G_3$ and assume that $G_2$
has $q \geq 2$ connected components $J_1, \ldots, J_q$. Observe that the total
order induced by $\sigma$ is in this case simply a concatenation of
total orders on $G_2$ and $G_3$, and the total order on $G_2$ must
refine the partial order induced by the edges of $G$, i.e. it is simply a
shuffle of total orders on $J_1, \ldots, J_q$.
Hence in the antisymmetrization
the operator $h \lambda_s$ is applied to
$$
\pm \big[\sum_{\sigma_1 \in \Sigma_{J_1}} (-1)^{\sigma_1} \sigma_1(a_{J_1}) \big] \# \ldots \#
\big[\sum_{\sigma_q \in \Sigma_{J_q}} (-1)^{\sigma_1} \sigma_q(a_{J_q})\big]
$$
where $\#$ stands for the shuffle product on the tensors and
$a_{J_i}$ is the ordered tensor product of elements in $J_i$.
Since $\mathcal{A}$ is graded commutative, the operations
$h\lambda_s$ vanish when applied to shuffle products by Theorem 12 in \cite{CG}
(in fact, we use the vanishing on the shuffles of the higher components of the
$A_\infty$ map $A \to \mathcal{A}$). Hence the terms of
the antisymmetrization of (\ref{merkulov}) which do not show up in
(\ref{induct}), sum up to zero, as required.
\end{proof}

\begin{figure}
\begin{center}
  \scalebox{0.8}{\includegraphics{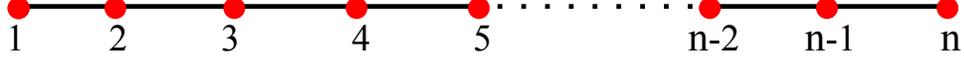}}
 \caption{Linear graph $G$.}
\end{center}
\label{Linear}
\end{figure}

\begin{example}
  Suppose that $G$ is a linear graph as in Figure 2. Suppose also that
$a_1, \ldots, a_n \in A$ are such that $m_{i} (a_p, \ldots, a_{p+1-i}) = 0$
for all $i \geq 2$ and $1 \leq p \leq k-i+1$. Then by corollary to
Proposition \ref{standard} and Proposition \ref{massey} the product
$a_1 \otimes \ldots \otimes a_n$
represents a class in $E_k$ and $\partial_{n-1}(a_1 \otimes \ldots \otimes a_n)$
is represented, up to sign,
by $m_n(a_1, \ldots, a_n)$. Observe that by
\cite{LPWZ} and \cite{JL} under the same assumptions the $n$-fold
Massey product of $a_1, \ldots, a_n$ is well defined, and represented
(as a coset),
up to a sign, by the value $m_n(a_1, \ldots, a_n)$. Thus, considering
sub-paths in a complete graph on $n$ vertices we give a concrete
formulation (and proof) of the conjecture at the bottom of page 429
of \cite{BG} that ``the higher differentials ... are determined by
higher-order Massey products". The $n=3$ case of this observation
was proved earlier in \cite{FT}.

\end{example}

\subsection{Degeneration}

\begin{prop}
\label{degeneration}
Assume that for a choice of homotopy $h$ all higher $A_\infty$
 products
vanish $\mu_{i} = 0$, $i \geq 3$ (e.g. $M$ is Kahler). Then the
spectral sequence degenerates at the $E_2$ term: $\partial_t = 0$
for $t \geq 2$. In general, if $k \leq n-1$ is the maximal length
of a sub-tree in $G$, then the spectral sequence
degenerates at the $E_{k+1}$ term: $\partial_t =0$ for $t \geq k+1$.
\end{prop}
\begin{proof} For the first part, by homological perturbation theory,
 there is
an $A_\infty$-map $A \to \mathcal{A}$ which induces a quasi-isomorphism
of $A_\infty$-algebras. This can be encoded by a single $R$-linear map
$B(A) \to \mathcal{A}$ such that the canonical multiplicative extension $\Omega(B(A))
\to \mathcal{A}$ is a quasi-isomorphism of $DG$-algebras,
where
$\Omega$ and $B$ are the cobar and bar constructions, respectively.
But since the higher products vanish, the natural map
$\Omega(B(A)) \to A$ is also a quasi-isomorphism of DG-algebras.
Since the differential on
$A$ is zero, the spectral sequence of
$C_A(G)$ degenerates at the $E_2$ term.

For the second part, assume that $x \in E_1^{p, q}$ represents a
class in $E_k^{p, q}$ and let us show that $\partial_t (x) = 0$
for $t > k$.
We can assume that $x$ is a linear combination of
elements in $e_s \cdot A^{\otimes l(s)}$ with $|s| = p$ and fixed
$l(s) = l$. Since $x$ represents an element in $E_k$, the system
of equations (\ref{system}) on $x_ \in E_1^{p+j-1, q+1-j}$
with $j =2, \ldots, k-1$, admits a solution. From our results
on the operators $d_i$ in the previous subsection, we can assume that
$x_j$ is a linear combination of terms $e_s \cdot A^{\otimes l(s)}$
where $s$ contains a subtree of length $j-1$. By the same result
$d_i(x_j) = 0$ if $i+j -1 > k$. Therefore $\partial_t(x) = 0$
for $t > k$. \end{proof}

\begin{remark}
When $G$ is a complete graph on $n$ vertices, the maximal subtree
length is $k=n-1$. Our result  $\partial_t =0$ for $t \geq n$ is
a little weaker than Proposition 4.2 in \cite{FT} which asserts that
$\partial_{n-1} = 0$ as well.
\end{remark}

\noindent
Addresses:

\medskip
\noindent
VB: Department of Mathematics, 340 Rowland Hall, UC Irvine, Irvine CA, 92617;
email: vbaranov@uci.edu; fax: +1-949-824-7993

\medskip
\noindent
RS: Department of Mathematics, University of Pennsylvania, 209 South 33rd Street,
Philadelphia PA, 19104-6395;
email: radmilas@math.upenn.edu

\end{document}